\pdfoutput=1
\RequirePackage{ifpdf}
\ifpdf % We~are running pdfTeX in pdf mode
\documentclass[pdftex]{sigma}
\else
\documentclass{sigma}
\fi

\numberwithin{equation}{section}

\newtheorem{Theorem}{Theorem}[section]
\newtheorem*{Theorem*}{Theorem}
\newtheorem{Corollary}[Theorem]{Corollary}
\newtheorem{Lemma}[Theorem]{Lemma}
\newtheorem{Proposition}[Theorem]{Proposition}
 { \theoremstyle{definition}
\newtheorem{Definition}[Theorem]{Definition}
\newtheorem{Notation}[Theorem]{Notation}

 }

\begin{document}
%\allowdisplaybreaks

\newcommand{\arXivNumber}{2503.04547}

\renewcommand{\PaperNumber}{053}

\FirstPageHeading

\ShortArticleName{Some Spherical Function Values for Hook Tableaux Isotypes and Young Subgroups}

\ArticleName{Some Spherical Function Values\\ for Hook Tableaux Isotypes and Young Subgroups}

\Author{Charles F.~DUNKL}

\AuthorNameForHeading{C.F.~Dunkl}

\Address{Department of Mathematics, University of Virginia, \\ PO Box 400137, Charlottesville VA 22904-4137, USA}
\Email{\href{mailto:cfd5z@virginia.edu}{cfd5z@virginia.edu}}
\URLaddress{\url{https://uva.theopenscholar.com/charles-dunkl}}

\ArticleDates{Received March 12, 2025, in final form June 30, 2025; Published online July 08, 2025}

\Abstract{A Young subgroup of the symmetric group $\mathcal{S}_{N}$, the permutation group of $\{ 1,2,\dots,N\} $, is generated by a subset of the adjacent
transpositions $\{ ( i,i+1) \mid 1\leq i<N\}$. Such a~group is realized as the stabilizer $G_{n}$ of a monomial \smash{$x^{\lambda}$ $\big({=}\,x_{1}^{\lambda_{1}}x_{2}^{\lambda_{2}}\cdots x_{N}^{\lambda_{N}}\big)$} with \smash{${\lambda=\bigl( d_{1}^{n_{1}},d_{2}^{n_{2}}, \dots,d_{p}^{n_{p}}\bigr)} $} (meaning $d_{j}$ is repeated $n_{j}$ times, $1\leq j\leq p$, and $d_{1}>d_{2}>\dots>d_{p}\geq0$), thus is isomorphic to the direct product $\mathcal{S}_{n_{1}}\times\mathcal{S}_{n_{2}} \times\cdots\times\mathcal{S}_{n_{p}}$. The interval $\{ 1,2,\dots,N\} $ is a union of disjoint sets $I_{j}= \{ i\mid \lambda_{i}=d_{j} \} $. The orbit of $x^{\lambda}$ under the action of $\mathcal{S}_{N}$ (by permutation of coordinates) spans a module $V_{\lambda}$, the representation induced from the identity representation of $G_{n}$. The space $V_{\lambda}$ decomposes into a~direct sum of irreducible $\mathcal{S}_{N}$-modules. The spherical function is defined for each of these, it is the character of the module averaged over the group $G_{n}$. This paper concerns the value of certain spherical functions evaluated at a cycle which has no more than one entry in each interval $I_{j}$. These values appear in the study of eigenvalues of the Heckman--Polychronakos operators in the paper by V.~Gorin and the author~[arXiv:2412:01938]. In particular, the present paper determines the spherical function value for $\mathcal{S}_{N}$-modules of hook tableau type, corresponding to Young tableaux of shape \smash{$\bigl[ N-b,1^{b}\bigr]$}.}


\Classification{20C30; 43A90; 20B30}

\begin{flushright}
\begin{minipage}{83mm}
\it To the memory of G. de B. Robinson (1906--1992)\\ my first algebra professor
\end{minipage}
\end{flushright}

\section{Introduction}

There is a commutative family of differential-difference operators acting on
polynomials in $N$ variables whose symmetric eigenfunctions are Jack
polynomials. They are called Heckman--Polychronakos operators, defined by
$\mathcal{P}_{k}:=\sum_{i=1}^{N}( x_{i}\mathcal{D}_{i}) ^{k}$,
$k=1,2,\dots$, in terms of Dunkl operators
\[
\mathcal{D}_{i}f(
x) :=\frac{\partial}{\partial x_{i}}f( x) +\kappa
\sum_{j=1,j\neq i}^{N}\frac{f( x) -f( x(i,j)
) }{x_{i}-x_{j}};
\] $x(i,j) $ denotes $x$ with $x_{i}$ and
$x_{j}$ interchanged, and $\kappa$ is a fixed parameter, often satisfying
\smash{$\kappa>-\frac{1}{N}$} (see Heckman \cite{Heckman1991}, Polychronakos
\cite{Polychronakos1992}; these citations motivated the name given the
operators in~\cite{DunklGorin2024}). The symmetric group on $N$ objects, that
is, the permutation group of $\{ 1,2,\dots,N\} $, is denoted by
$\mathcal{S}_{N}$ and acts on~$\mathbb{R}[ x_{1},\dots,x_{N}] $
by permutation of the variables. Specifically, for a~polynomial $f(
x) $ and $w\in\mathcal{S}_{N}$ the action is $wf( x)
=f( xw) $, $( xw) _{i}=x_{w(i) }$, $1\leq
i\leq N$. This is a representation of $\mathcal{S}_{N}$. The operators
$\mathcal{P}_{k}$ commute with this action and thus the structure of
eigenfunctions and eigenvalues is strongly connected to the decomposition of
the space of polynomials into irreducible $\mathcal{S}_{N}$-modules. The
latter are indexed by partitions of $N$, that is, $\tau=( \tau
_{1},\dots,\tau_{\ell}) $ with $\tau_{i}\in\mathbb{N}$, $\tau_{1}\geq
\tau_{2}\geq\dots\geq\tau_{\ell}>0$ and \smash{$\sum_{i=1}^{\ell}\tau_{i}=N$}. The
corresponding module is spanned by the standard Young tableaux of shape $\tau
$. The general details are not needed here. The types of polynomial modules of
interest here are spans of certain monomials: for $\alpha\in\mathbb{Z}_{+}%
^{N}$, let \smash{$x^{\alpha}:=\prod_{i=1}^{N}x_{i}^{\alpha_{i}}$}. Suppose
$\lambda_{1}\geq\lambda_{2}\geq\dots\geq\lambda_{N}\geq0$, then set
$V_{\lambda}=\mathrm{span}_{\mathbb{F}}\bigl\{ x^{\beta}\mid \beta=w\lambda
,\, w\in\mathcal{S}_{N}\bigr\} $, that is, $\beta$ ranges over the permutations
of $\lambda$, and $\mathbb{F}$ is an extension field of $\mathbb{R}$
containing at least $\kappa$. The space $V_{\lambda}$ is invariant under the
action of $\mathcal{S}_{N}$. The eigenvector analysis of $\mathcal{P}_{k}$ is
based on the triangular decomposition $\mathcal{P}_{k}V_{\lambda}\subset
V_{\lambda}\oplus\sum_{\nu\prec\lambda}\oplus V_{\nu}$, where $\nu\prec
\lambda$ is the dominance order, \smash{$\sum_{i=1}^{j}\nu_{i}\leq\sum_{i=1}%
^{j}\lambda_{i}$} for $1\leq j\leq N$, and \smash{$\sum_{i=1}^{N}\nu_{i}=\sum
_{i=1}^{N}\lambda_{i}$}. Part of the analysis is to identify irreducible
$\mathcal{S}_{N}$-submodules of $V_{\lambda}$. This~depends on the number
of repetitions of values among $\{ \lambda_{i}\mid 1\leq i\leq N\} $.
To be precise let $\lambda=\bigl( d_{1}^{n_{1}},d_{2}^{n_{2}},\dots
,d_{p}^{n_{p}}\bigr) $ (that is, $d_{j}$ is repeated $n_{j}$ times, $1\leq
j\leq p$), with $d_{1}>d_{2}>\dots>d_{p}\geq0$ and $N=\sum_{i=1}^{p}n_{i}$.
Let $G_{\mathbf{n}}$ denote the stabilizer group of~$x^{\lambda}$, so that
$G_{\mathbf{n}}\cong\mathcal{S}_{n_{1}}\times\mathcal{S}_{n_{2}}%
\times\cdots\times\mathcal{S}_{n_{p}}$. The representation of $\mathcal{S}%
_{N}$ realized on $V_{\lambda}$ is the induced representation \smash{$\mathrm{ind}%
_{G_{\mathbf{n}}}^{\mathcal{S}_{N}}$}. This decomposes into irreducible $S_{N}%
$-modules and the number of copies (the multiplicity) of a particular isotype
$\tau$ in $V_{\lambda}$ is called a \emph{Kostka} number (see Macdonald~\mbox{\cite[p.~101]{Macdonald1995}}).

The operator $\mathcal{P}_{k}$ arose in the study of the Calogero--Sutherland
quantum system of $N$ identical particles on a circle with inverse-square
distance potential: the Hamiltonian is%
\[
\mathcal{H=-}\sum_{j=1}^{N}\biggl( \frac{\partial}{\partial\theta_{j}}\biggr)
^{2}+\frac{\kappa( \kappa-1) }{2}\sum_{1\leq i<j\leq N}\frac
{1}{\sin^{2}\bigl( \frac{1}{2}( \theta_{i}-\theta_{j}) \bigr)
};
\]
the particles are at $\theta_{1},\dots,\theta_{N}$ and the chordal distance
between two points is $\bigl\vert 2\sin\bigl( \frac{1}{2}( \theta
_{i}-\theta_{j}) \bigr) \bigr\vert $. By changing variables
$x_{j}=\exp\mathrm{i}\theta_{j}$, the Hamiltonian is transformed to
\[
\mathcal{H=}\sum_{j=1}^{N}\biggl( x_{j}\frac{\partial}{\partial x_{j}}\biggr)
^{2}-2\kappa( \kappa-1) \sum_{1\leq i<j\leq N}\frac{x_{i}x_{j}
}{( x_{i}-x_{j}) ^{2}}
\]
(for more details see Chalykh \cite[p.~16]{Chalykh2024}).

In \cite{DunklGorin2024}, Gorin and the author studied the eigenvalues of the
operators $\mathcal{P}_{k}$ restricted to submodules of $V_{\lambda}$ of given
isotype $\tau$. It turned out that if the multiplicity of the isotype $\tau$
in $V_{\lambda}$ is greater than one, then the eigenvalues are not rational in
the parameters and do not seem to allow explicit formulation. However, the sum
of all the eigenvalues (for any fixed $k$) can be explicitly found, in terms
of the character of $\tau$. In general, this may not have a relatively simple
form but there are cases allowing a closed form. The present paper carries
this out for hook isotypes, labeled by partitions of the form $\bigl[
N-b,1^{b}\bigr] $ (the Young diagram has a hook shape). The formula for the
sum is quite complicated with a number of ingredients. For given $(
n_{1},\dots,n_{p}) $ define the intervals associated with $\lambda$,
$I_{j}=\bigl[ \sum_{i=1}^{j-1}n_{i}+1,\sum_{i=1}^{j}n_{i}\bigr] $ for
$1\leq j\leq p$ (notation $[ a,b] :=\{ a,a+1,\dots
,b\} \subset\mathbb{N})$. The formula is based on considering cycles
$g_{\mathcal{A}}$ corresponding to subsets $\mathcal{A=}\{ a_{1}%
,\dots,a_{\ell}\} $ of $[ 1,p] $, which are of length
$\ell$ with exactly one entry from each interval~$I_{a_{j}}$. Any such cycle
can be used and the order of $a_{1},\dots,a_{\ell}$ does not matter. The
degrees~$d_{1},\dots,d_{p}$ enter the formula in a shifted way:%
\[
\widetilde{d}_{i}:=d_{i}+\kappa( n_{i+1}+n_{i+2}+\cdots+n_{p}),\qquad
1\leq i\leq p.
\]
Let $h_{m}^{\mathcal{A}}:=h_{m}\bigl( \widetilde{d}_{a_{1}},\widetilde
{d}_{a_{2}},\dots,\widetilde{d}_{a_{\ell}}\bigr) $, the complete symmetric
polynomial of degree $m$ (the generating function is $\sum_{k\geq0}%
h_{k}( c_{1},c_{2},\dots,c_{q}) t^{k}=\prod_{i=1}^{q}(
1-c_{i}t) ^{-1}$, see \cite[p.~21]{Macdonald1995}). In
\cite{DunklGorin2024}, we used an~``averaged
character'' (spherical function, in the present paper).
Denote the character of the representation~$\tau$ of $\mathcal{S}_{N}$ by
$\chi^{\tau}$, then
\[
\chi^{\tau}[ \mathcal{A};\mathbf{n}] :=\frac{1}{\#G_{\mathbf{n}}
}\sum_{h\in G_{\mathbf{n}}}\chi^{\tau}( g_{\mathcal{A}} h) ,
\]
where $g_{\mathcal{A}}$ is an $\ell$-cycle labeled by $\mathcal{A}$ as above,
and $\#G_{\mathbf{n}}=\prod_{i=1}^{p}n_{i}!$. In formula \eqref{egvsum0}, the
inner sum is over $k$-subsets of $[ 1,p] $, the list of labels of
the partition $\{ I_{1},\dots,I_{p}\} $, so a typical subset is
$\{ a_{1},a_{2},\dots,a_{k}\} $, $g_{\mathcal{A}}$ is a cycle
with one entry in each $I_{a_{i}}$ and \smash{$h_{k+1-\ell}^{\mathcal{A}}%
=h_{k+1-\ell}\bigl( \widetilde{d}_{a_{1}},\widetilde{d}_{a_{2}}%
,\dots,\widetilde{d}_{a_{\ell}}\bigr)$}.

Now suppose the multiplicity of $\tau$ in $V_{\lambda}$ is $\mu$, then there
are $\mu\dim\tau$ eigenfunctions and eigenvalues of $\mathcal{P}_{k}$, and the
sum of all these eigenvalues is \cite[Theorem~5.4]{DunklGorin2024}
\begin{equation}
\dim\tau\sum_{\ell=1}^{\min( k+1,p) }( -\kappa)
^{\ell-1}\sum_{\mathcal{A}\subset[ 1,p] ,\, \#\mathcal{A}=\ell}
\chi^{\tau}[ \mathcal{A};\mathbf{n}] h_{k+1-\ell}^{\mathcal{A}
}\prod_{i\in\mathcal{A}}n_{i}!. \label{egvsum0}
\end{equation}

The main result of the present paper is to establish an explicit formula for
$\chi^{\tau}[ \mathcal{A};\mathbf{n}] $ with $\tau=\bigl[
N-b,1^{b}\bigr] $. Since the order of the factors of $G_{\mathbf{n}}$ in the
character calculation does not matter (by the conjugate invariance of
characters), it will suffice to take $\mathcal{A}=\{ 1,2,\dots,\ell\} $, for~${2\leq\ell\leq p}$. In the following, $e_{i}$ denotes the
elementary symmetric polynomial of degree $i$ and the Pochhammer symbol is
$( a) _{n}=\prod_{i=1}^{n}( a+i-1) $. The combined
main results are the following.

\begin{Theorem}
Let $m:=p-b-1$, $\ell\leq p$ and $\mathcal{A=}\{ 1,2,\dots,\ell\} $,
then
\begin{gather}
\chi^{\tau}[ \mathcal{A};\mathbf{n}] =\frac{1}{\prod_{i=1}^{\ell
}n_{i}}\label{big1}\\ \hphantom{\chi^{\tau}[ \mathcal{A};\mathbf{n}] =}{}
\times\Biggl\{ \sum_{k=0}^{\min( m,\ell) }\! \frac{(
b+1) _{m-k}}{( m-k) !}e_{\ell-k}( n_{1}
-1,n_{2}-1,\dots,n_{\ell}-1) +( -1) ^{\ell+1}\frac{(
b-\ell+1) _{m}}{m!}\!\Biggr\} \nonumber\\
\hphantom{\chi^{\tau}[ \mathcal{A};\mathbf{n}] }{}
=\binom{b+m}{b}+
{\sum\limits_{i=1}^{\min( b,\ell-1) }}
%EndExpansion
( -1) ^{i}\binom{b+m-i}{b-i}e_{i}\biggl( \frac{1}{n_{1}}%
,\frac{1}{n_{2}},\dots,\frac{1}{n_{\ell}}\biggr). \label{big2}
\end{gather}
\end{Theorem}

These results come from Theorems \ref{thmlb2}, \ref{thmlb3} and Proposition~\ref{eqbig23}. There are no nonzero $G_{\mathbf{n}}$-invariants if $m<0$ (as
will be seen).

In Section \ref{SphFun}, we present general background on spherical functions,
harmonic analysis, and subgroup invariants for finite groups. Section
\ref{AltPol} concerns alternating polynomials, which span a~module of isotype
$\bigl[ N-b,1^{b}\bigr] $. There is the definition of sums of alternating
polynomials which make up a basis for $G_{\mathbf{n}}$-invariants. The main
results, proving formula (\ref{big1}) for the case $p=b+1$ are in Section
\ref{peqb+1}, and for the cases~$p\geq b+2$ are in Section \ref{pgtb1}, with
subsections for~$p=b+2$ and~$p>b+2.$ The details are of increasing
technicality. Section \ref{formb2} deduces formula~(\ref{big2}) from~(\ref{big1}).
Some specializations of the formulas are discussed.

\section{Spherical functions}\label{SphFun}

Suppose the representation $\tau$ of $\mathcal{S}_{N}$ is realized on a linear
space $V$ furnished with an $\mathcal{S}_{N}$-invariant inner product, then
there is an orthonormal basis for $V$ in which the restriction to
$G_{\mathbf{n}}$ decomposes as a direct sum of irreducible representations of
$G_{\mathbf{n}}$. Suppose the multiplicity of $1_{G_{\mathbf{n}}}$ is $\mu$,
and the basis is chosen so that for $h\in G_{\mathbf{n}}$%
\[
\tau( h) =%
\begin{bmatrix}
I_{\mu} & O & \dots & O\\
O & T^{(1) }( h) & \dots & O\\
\dots & \dots & \dots & \dots\\
O & O & \dots & T^{( r) }( h)
\end{bmatrix}
,
\]
where $T^{(1) },\dots,T^{( r) }$ are irreducible
representations of $G_{\mathbf{n}}$ not equivalent to $1_{G_{\mathbf{n}}}$
(see \cite[Sections~3.6 and~10]{Ceccherini-&ST2008} for details). For $g\in
\mathcal{S}_{N}$, denote the matrix of $\tau( g) $ with respect to
the basis by~$\tau_{i,j}( g) $; since $\tau( h_{1}
gh_{2}) =\tau( h_{1}) \tau( g) \tau(
h_{2}) $, we find
\[
\frac{1}{( \#G_{\mathbf{n}}) ^{2}}\sum_{h_{1},h_{2}\in
G_{\mathbf{n}}}\tau_{i,j}( h_{1}gh_{2}) =
\begin{cases}
\tau_{i,j}( g), & 1\leq i,j\leq\mu,\\
0,& \mathrm{else}.
\end{cases}
\]
Then $\{ \tau_{i,j}\mid 1\leq i,j\leq\mu\} $ is a basis for the
$G_{\mathbf{n}}$-bi-invariant elements of $\mathrm{span}\{ \tau_{k\ell
}\} $. The spherical function for the isotype $\tau$ and subgroup
$G_{\mathbf{n}}$ is defined by
\[
\Phi^{\tau}( g) :=\sum_{i=1}^{\mu}\tau_{ii}( g),\qquad g\in\mathcal{S}_{N}%
\]
\big(our notation for $\chi^{\tau}[ \mathcal{A};\mathbf{n}] $\big).
Sometimes the term ``spherical function'' is reserved for Gelfand pairs where
the multiplicity $\mu=1.$ The character of $\tau$ is $\chi^{\tau}(
g) = {\rm tr}(\tau( g) ) $, then
\[
\Phi^{\tau
}( g) =\frac{1}{\#G_{\mathbf{n}}}\sum_{h\in G_{\mathbf{n}}}%
\chi^{\tau}( hg).
\]

The symmetrization operator acting on $V$ is ($\zeta\in V$)%
\[
\rho\zeta:=\frac{1}{\#G_{\mathbf{n}}}\sum_{h\in G_{\mathbf{n}}}\tau(
h) \zeta.
\]
The operator $\rho$ is a self-adjoint projection.

Suppose there is an orthogonal subbasis $\{ \psi_{i}\mid 1\leq i\leq
\mu\} $ for $V$ which satisfies $\tau( h) \psi_{i}%
=\psi_{i}$ for $h\in G_{\mathbf{n}}$ (thus $\rho\psi_{i}=\psi_{i}$) and $1\leq
i\leq\mu$, then the matrix element $\tau_{i,i}( g) =\langle
\tau( g) \psi_{i},\psi_{i}\rangle /\langle \psi
_{i},\psi_{i}\rangle $ and the spherical function \smash{$\Phi^{\tau}(
g) =
{\sum_{i=1}^{\mu}}
\frac{1}{\langle \psi_{i},\psi_{i}\rangle }\langle
\tau( g) \psi_{i},\psi_{i}\rangle $}. We will produce a
formula for $\Phi^{\tau}( g) $ which works with a non-orthogonal
basis of $G_{\mathbf{n}}$-invariant vectors $\{ \xi_{i}\mid 1\leq i\leq
\mu\} $ in $V.$ The Gram matrix $M$ is given by $M_{ij}:=\langle
\xi_{i},\xi_{j}\rangle $. For $g\in\mathcal{S}_{N}$, let $T(
g) _{ij}:=\langle \tau( g) \xi_{j},\xi_{i}\rangle $.

\begin{Lemma}
$\Phi^{\tau}( g) =\mathrm{tr}\bigl( T( g)
M^{-1}\bigr) $.
\end{Lemma}

\begin{proof}
Suppose $\{ \zeta_{i}\mid 1\leq i\leq\mu\} $ is an orthonormal basis
for the $G_{\mathbf{n}}$-invariant vectors in $V$, then there is a (change of
basis) matrix $[ A_{ij}] $ such that $\zeta_{i}=\sum_{j=1}^{\mu
}A_{ji}\xi_{j}$ and
\begin{align*}
&\langle \tau( g) \zeta_{i},\zeta_{i}\rangle
=\Biggl\langle \sum_{j=1}^{\mu}\tau( g) A_{ji}\xi_{j},\sum
_{k=1}^{\mu}A_{ki}\xi_{k}\Biggr\rangle =\sum_{j,k}A_{ji}A_{ki}T(
g) _{kj},\\
&\Phi^{\tau}( g) =\sum_{i}\sum_{j,k}A_{ji}A_{ki}T(
g) _{kj}=\sum_{j,k}( AA^{\ast}) _{jk}T( g)
_{kj}={\rm tr}(( AA^{\ast}) T( g) ).
\end{align*}
Furthermore,
\begin{align*}
\delta_{ij} =\langle \zeta_{i},\zeta_{j}\rangle =\sum
_{k,r}A_{ki}A_{rj}\langle \xi_{k},\xi_{r}\rangle =\sum_{k,r}
A_{ki}A_{rj}M_{kr}=( A^{\ast}MA) _{ij}
\end{align*}
and $A^{\ast}MA=I$, $M=( A^{\ast}) ^{-1}A^{-1}=( AA^{\ast}) ^{-1}$.
\end{proof}

\begin{Corollary}
Suppose $\rho\tau( g) \xi_{i}=\sum_{j=1}^{\mu}B_{ji}(
g) \xi_{j}$, $1\leq i,j\leq\mu$, then $\Phi^{\tau}( g)
=\mathrm{tr}( B( g) ) $.
\end{Corollary}

\begin{proof}
The expansion holds because $\{ \xi_{i}\mid 1\leq i\leq\mu\} $ is a
basis for the invariants. Then
\begin{align*}
T( g) _{ij}=\langle
\tau( g) \xi_{j},\xi_{i}\rangle =\langle \rho
\tau( g) \xi_{j},\xi_{i}\rangle =\Biggl\langle \sum
_{k=1}^{\mu}B_{kj}( g) \xi_{k},\xi_{i}\Biggr\rangle =\sum
_{k=1}^{\mu}B_{kj}( g) M_{ki}=\bigl( M^{\mathsf{T}}B( g)
\bigr) _{ij}
\end{align*} and $\mathrm{tr}\bigl( T( g) M^{-1}\bigr)
=\mathrm{tr}\bigl( M^{\mathsf{T}}B( g) M^{-1}\bigr) =\mathrm{tr}(
B( g) ) $ \big(note $M^{\mathsf{T}}=M$\big).
\end{proof}

We will use the method provided by the Corollary to determine $\Phi^{\tau
}( g) $. This formula avoids computing $T( g) $ and
the inverse \smash{$M^{-1}$} of the Gram matrix. When the multiplicity $\mu=1$, the
formula simplifies considerably: there is one invariant $\psi_{1}$, $T(
g) =\langle g\psi_{1},\psi_{1}\rangle $ and $M=[
\langle \psi_{1},\psi_{1}\rangle ] $ so that \smash{$\Phi^{\tau
}( g) =\frac{\langle g\psi_{1},\psi_{1}\rangle
}{\langle \psi_{1},\psi_{1}\rangle }$} (or $=c$ if $\rho g\psi
_{1}=c\psi_{1}$). We are concerned with computing the spherical function at a
cycle $g$ of length $\ell$ with no more than one entry from each interval
$I_{j}$ , where the factor $\mathcal{S}_{n_{j}}$ acts only on $I_{j}$. The
idea is to specify the $G_{\mathbf{n}}$-invariant polynomials $\xi$, the
effect of an $\ell$-cycle on each of these, and then compute the expansion of
$\rho g\xi$ in the invariant basis.

\section[Coordinate systems and invariant sums of alternating polynomials]{Coordinate systems and invariant sums\\ of alternating
polynomials}\label{AltPol}

To clearly display the action of $G_{\mathbf{n}}$, we introduce a modified
coordinate system. Replace
\[
( x_{1},x_{2},\dots,x_{N}) =\bigl( x_{1}^{(1)
},\dots,x_{n_{1}}^{(1) },x_{1}^{(2) }%
,\dots,x_{n_{2}}^{(2) },\dots,x_{1}^{( p)
},\dots,x_{n_{p}}^{( p) }\bigr) ,
\]
that is, \smash{$x_{i}^{( j) }$} stands for $x_{s}$ with \smash{$s=\sum
_{i=1}^{j-1}n_{i}+i$}. We use \smash{$x_{\ast}^{(i) }$}, \smash{$x_{>}^{(
i) }$} to denote a generic \smash{$x_{j}^{(i) }$} with ${1\leq
j\leq n_{i}}$, respectively, $2\leq j\leq n_{i}$. In the sequel, $g$ denotes the
cycle \smash{$\bigl( x_{1}^{(1) },x_{1}^{(2) }
,\dots,x_{1}^{(\ell) }\bigr) $}. Throughout, denote $m:=p-b-1$.

\begin{Notation}
For $0\leq j\leq\ell$, denote the elementary symmetric polynomial of degree $j$
in the variables $n_{1}-1,n_{2}-1,\dots,n_{\ell}-1$ by $e_{j}( n_{\ast
}-1) $. Set \smash{$\pi_{\ell}:=\prod_{i=1}^{\ell}n_{i}$} and \smash{$\pi_{p}%
:=\prod_{i=1}^{p}n_{i}$}. For integers $i\leq j$, the interval $\{
i,i+1,\dots,j\} \subset\mathbb{N}$ is denoted by $[ i,j]
$.
\end{Notation}

\begin{Definition}
The action of the symmetric group $\mathcal{S}_{N}$ on polynomials $P(
x) $ is given by $wP( x) =P( xw) $ and
$( xw) _{i}=x_{w(i) }$, $w\in\mathcal{S}_{N}$, $1\leq
i\leq N$.
\end{Definition}

Note $( x( vw) ) _{i}=( xv) _{w(
i) }=x_{v( w(i) ) }=x_{vw(i)
}$, $vwP( x) =( wP) ( xv) =P(
xvw) $. The projection onto $G_{\mathbf{n}}$-invariant polynomials is
given by%
\[
\rho P( x) =\frac{1}{\#G_{\mathbf{n}}}%
{\sum\limits_{h\in G_{\mathbf{n}}}}
P( xh).
\]

We use the polynomial module of isotype $\bigl[ N-b,1^{b}\bigr] $ with the
lowest degree. This module is spanned by alternating polynomials in $b+1$ variables.

\begin{Definition}
For $x_{i_{1}},x_{i_{2}},\dots x_{i_{b+1}}$, let
\[
\Delta\bigl( x_{i_{1}},x_{i_{2}},\dots x_{i_{b+1}}\bigr) :=\prod
\limits_{1\leq j<k\leq b+1}\bigl( x_{i_{j}}-x_{i_{k}}\bigr).
\]

\end{Definition}

\begin{Lemma}\label{sumdelt}
Suppose $x_{1},\dots,x_{b+2}$ are arbitrary variables
and $f_{j}:=\Delta( x_{1},x_{2,},\dots,\widehat{x_{j}%
},\dots,x_{b+2}) $ denotes the alternating polynomial when $x_{j}$ is
removed from the list, then $\sum_{j=1}^{b+2}( -1) ^{j}f_{j}=0$.
\end{Lemma}

\begin{proof}
Let $F( x) :=\sum_{j=1}^{b+2}( -1) ^{j}f_{j}$.
Suppose $1\leq i<b+2$ and $( i,i+1) $ is the transposition of
$x_{i}$ and $x_{i+1}$, then $( i,i+1) f_{j}=-f_{j}$ if $j\neq i, i+1$,
$( i,i+1) f_{i}=f_{i+1}$ and $( i,i+1) f_{i+1}=f_{i}$.
Thus ${( i,i+1) F( x) =-F( x) }$ and this
implies $F( x) $ is divisible by the alternating polynomial in
$x_{1},\dots,x_{b+2}$ which is of degree \smash{$\frac{1}{2}( b+2)
( b+1) $}, but $F$ is of degree \smash{$\leq\frac{1}{2}b(
b+1) $} and hence~${F( x) =0}$.
\end{proof}

We briefly discuss the relation between hook tableaux and alternating
polynomials. The~irreducible representation of $\mathcal{S}_{N}$ corresponding
to the partition $\bigl[ N-b,1^{b}\bigr] $ has the basis of standard Young
tableaux of this shape. For a given tableau $T$ and an entry $i$ (with $1\leq
i\leq N$) the content $c( i,T) :=\operatorname{col}(
i,T) -\operatorname{row}( i,T) $ (the labels of the
column, row of $T$ containing $i$). The Jucys--Murphy elements $\omega
_{j}:=\sum_{i=1}^{j-1}(i,j) $, $1\leq j\leq N$, mutually commute
and satisfy $\omega_{j}T=c( j,T) T$; this is the defining
property of the representation. There is a general identity $\omega
_{j+1}=\sigma_{j}\omega_{j}\sigma_{j}+\sigma_{j}$ where $\sigma_{j}:=(
j,j+1) $ for $1\leq j<N$. The tableau $T_{0}$ with first column
containing $\{ 1,2,\dots,b+1\} $ satisfies $c(
i,T_{0}) =1-i$ for $1\leq i\leq b+1$ and $=i-b-1$ for $b+2\leq i\leq
N$. We use the notation of Lemma~\ref{sumdelt} so that $f_{b+2}:=\Delta( x_{1},\dots,x_{b+1}) $.

\begin{Proposition}
$\omega_{i}f_{b+2}=c( i,T_{0}) f_{b+2}$ for $1\leq i\leq N$.
\end{Proposition}

\begin{proof}
Since $(i,j) f_{b+2}=-f_{b+2}$ for $1\leq i<j\leq b+1$, it
follows that $\omega_{i}f_{b+2}=-( i-1) f_{b+2}$ for $1<i\leq
b+1,$ while $\omega_{1}=0$ implies $\omega_{1}f_{b+2}=0=c(
1,T_{0}) f_{b+2}.$ The claim $\omega_{b+2}f_{b+2}=f_{b+2}$ needs more
detail. Consider the term for $( i,b+2) $ in $\omega_{b+2}%
f_{b+2}$ (for $1\leq i\leq b+1$)%
\[
( i,b+2) \Delta( x_{1},\dots x_{b+1})
=\Delta\bigl( x_{1},\dots,\overset{(i) }{x_{b+2}},\dots, x_{b+1}\bigr) =( -1) ^{b+1-i}f_{i},
\]
the sign comes from moving $x_{b+2}$ by $b+1-i$ adjacent transpositions to the
last argument of~$\Delta( \ast) $. Thus%
\[
\sum_{i=1}^{b+1}( i,b+2) f_{b+2}=\sum_{i=1}^{b+1}(
-1) ^{b+1-i}f_{i}=f_{b+2}%
\]
by the lemma \big(multiply each term by $( -1) ^{b+1}$\big). Now let
$i\geq b+2$ and suppose $\omega_{i}f_{b+2}=( i-b-1) f_{b+2}$, then
$\omega_{i+1}f_{b+2}=( \sigma_{i}\omega_{i}\sigma_{i}+\sigma_{i})
f_{b+2}=( \omega_{i}+1) f_{b+2}=( i-b) f_{b+2}$. By
induction, $\omega_{i}f_{b+2}=( i-b-1) f_{b+2}$ for $b+2\leq i\leq
N$ and this completes the proof.
\end{proof}

The $\mathcal{S}_{N}$-module spanned by $f_{b+2}$ is of isotype $\bigl[
N-b,1^{b}\bigr] $.

Usually $\mathbf{x}$ denotes a $( b+1) $-tuple as a generic
argument of $\Delta$.

\begin{Definition}
Suppose $\mathbf{x=}\bigl( x_{i_{1}}^{( j_{1}) },x_{i_{2}%
}^{( j_{2}) },\dots,x_{i_{b+1}}^{( j_{b+1})
}\bigr) $, then
\[
\mathcal{L}( \mathbf{x}) :=( j_{1}%
,j_{2},\dots,j_{b+1}) \qquad \text{and} \qquad \Delta( \mathbf{x})
:=\Delta\bigl( x_{i_{1}}^{( j_{1}) },x_{i_{2}}^{(
j_{2}) },\dots,x_{i_{b+1}}^{( j_{b+1}) }\bigr).
\]
\end{Definition}

The arguments in $\mathcal{L}( \mathbf{x}) $ can be assumed to be
in increasing order, up to a change in sign of~$\Delta( \mathbf{x}%
) $ (for example, if $\sigma$ is a transposition, then $\Delta(
\mathbf{x}\sigma) =-\Delta( \mathbf{x}) $).

\begin{Proposition}
\label{Lxh=Lx}If $h\in G_{\mathbf{n}}$, then $\mathcal{L}( \mathbf{x}%
h) =\mathcal{L}( \mathbf{x}) $, and%
\[
\rho\Delta( \mathbf{x}) =
{\prod_{r=1}^{b+1}}
n_{j_{r}}^{-1}\sum\{ \Delta( \mathbf{y}) \mid\mathcal{L} ( \mathbf{y}) =\mathcal{L}( \mathbf{x}) \}.
\]
\end{Proposition}

\begin{proof}The second statement follows from the multiplicative property of $\rho$ and
from
\[
\frac{1}{n_{j}!}\sum_{h\in\mathcal{S}_{n_{j}}}f\bigl( x_{\ast}^{(
j) }h\bigr) =\frac{1}{n_{j}}\sum_{i=1}^{n_{j}}f\bigl( x_{i}^{(
j) }\bigr),
\]
where \smash{$x_{\ast}^{( j) }$} denotes an arbitrary \smash{$x_{i}^{(
j) }$}, and $( n_{j}-1) !$ elements $h$ fix \smash{$x_{i}^{(
j) }$}.
\end{proof}

It follows from Lemma~\ref{sumdelt} that a basis for the $G_{\mathbf{n}}
$-invariant polynomials is generated from $\Delta( \mathbf{x}) $
with \smash{$\mathbf{x=}\bigl( x_{i_{1}}^{( j_{1}) },x_{i_{2}}^{(
j_{2}) },\dots,x_{i_{b+1}}^{( p) }\bigr) $}, that is, the
last coordinate is in $I_{p}$.

In the following, we specify invariants by the indices omitted from
$\mathcal{L}( \mathbf{x}) $; this is actually more convenient.

\begin{Definition}
Suppose $S\subset[ 1,p-1] $ with $\#S=m$, then define the
invariant polynomial
\[
\xi_{S}:=\sum\bigl\{ \Delta( \mathbf{x}) \mid \mathbf{x}=\bigl(
\dots, x_{i_{j}}^{( j) },\dots,x_{i_{p}}^{( p)
}\bigr) ,\,j\notin S\bigr\}.
\]
\end{Definition}

That is, the coordinates of $\mathbf{x}$ have $b$ distinct indices from
$[ 1,b+m] \setminus S$. The basis has cardinality \smash{$\mu
=\binom{b+m}{b}$} (see \cite[p.~105, Example 2\,(b)]{Macdonald1995}). The underlying
task is to compute the coefficient~$B_{S,S}( g) $ in $\rho
g\xi_{S}=\sum_{S^{\prime}}B_{S^{\prime},S}( g) \xi_{S^{\prime}}$.
This requires a decomposition of $\xi_{S}$.

\begin{Definition}
\label{defXE}Suppose $S\subset[ 1,p-1] $ with $\#S=m$ and
$E\subset( [ 1,\ell] \cup\{ p\} )
\setminus S$ say $\mathbf{x}\in X_{S,E}$ if~$j\in E$ implies \smash{$\mathbf{x}%
_{j}=x_{1}^{( j) }$}, $j\in[ \ell+1,p] \setminus S$
implies \smash{$\mathbf{x}_{j}=x_{k}^{( j) }$} with $1\leq k\leq n_{j}$
and $j\in[ 1,\ell] \setminus( S\cup E) $ implies
\smash{$\mathbf{x}_{j}=x_{k}^{( j) }$} with $2\leq k\leq n_{j}$ (consider
$\mathbf{x}$ as a $( b+1) $-tuple indexed by $[
1,b+m] \setminus S\cup\{ p\} $). Furthermore, let
\[
\phi_{S,E}:=\sum\{ \Delta( \mathbf{x}) \mid \mathbf{x}\in
X_{S,E}\}.
\]
\end{Definition}

Thus $\xi_{S}=\sum_{E}\phi_{S,E}$ and we will analyze $\rho g\phi_{S,E}$. It
turns out only a small number of sets~$E$ allow $\rho g\phi_{S,E}\neq0$, and
an even smaller number have a nonzero coefficient $B_{S,( S,E) }$
in the expansion $\rho g\phi_{S,E}=\sum_{S^{\prime}}B_{S^{\prime},(
S,E) }\xi_{S^{\prime}}$, namely $\varnothing$ (the empty set), $[
1,\ell] $ and $[ 1,\min S-1] \cup\{ p\} $.
Part of the discussion is to show this list is exhaustive.

\begin{Lemma}
\label{zeroD} For $\rho\Delta( \mathbf{x}) \neq0$, it is necessary
that there be no repetitions in $\mathcal{L}( \mathbf{x})$.
\end{Lemma}

\begin{proof}
Suppose $j_{a}=j_{b}=k,$ and \smash{$\Delta\bigl( \dots,x_{i_{a}}^{(
k) },\dots,x_{i_{b}}^{( k) },\dots\bigr) $} appears in
the sum; we can assume $b=a+1$ by rearranging the variables, possibly
introducing a sign factor. If $i_{a}=i_{b}$, then~${\Delta( \mathbf{x}%
) =0}$, else \smash{$\Delta\bigl( \dots,x_{i_{b}}^{( k)
},x_{i_{a}}^{( k) },\dots\bigr) $} also appears, by the action
of the transposition $( i_{a},i_{b}) \in\mathcal{S}_{n_{k}}$
These~two terms cancel out because $\Delta( \mathbf{x}(
i_{a},i_{b}) ) =-\Delta( \mathbf{x}) $.
\end{proof}

If some $\mathbf{x}$ has coordinates \smash{$x_{1}^{(i) }$},
\smash{$x_{>}^{( i+1) }$} and $i<\ell$, then $\rho g\Delta(
\mathbf{x}) =0$ by the lemma since \smash{$\mathbf{x}g=\bigl( \dots
,x_{1}^{( i+1) },x_{>}^{( i+1) },\dots\bigr) $}.
This strongly limits the sets $E$ allowing $\rho\Delta( \mathbf{x}%
g) \neq0$.

\section[The case p=b+1]{The case $\boldsymbol{p=b+1}$}\label{peqb+1}

This is the least complicated situation and introduces some techniques used
later. Here $\mathcal{L}( \mathbf{x}) =( 1,2,\dots
,b,b+1)$. There is just one $G_{\mathbf{n}}$-invariant polynomial (up
to scalar multiplication):%
\[
\psi:=\sum\bigl\{ \Delta\bigl( x_{i_{1}}^{(1) },x_{i_{2}
}^{(2) },\dots,x_{i_{p}}^{( b+1) }\bigr) \mid 1\leq
i_{1}\leq n_{1},\,1\leq i_{2}\leq n_{2},\,\dots,\, 1\leq i_{b+1}\leq n_{b+1}
\bigr\}.
\]
In Definition~\ref{defXE}, the set $S=\varnothing$ and we write $\phi_{E}$ for
$\phi_{\varnothing,E}$.

\begin{Proposition}
%\label{gEzero}
If $1\leq\#E<\ell$, then $\rho g\phi_{E}=0$.
\end{Proposition}

\begin{proof}
Suppose there are indices $k$, $k+1$ with $k\in E$, $k<\ell$ and $k+1\notin E$. If
$\Delta( \mathbf{x}) $ is one of the summands of $\phi_{E}$, then
\smash{$\mathbf{x=}\bigl( \dots,x_{1}^{( k) },x_{>}^{(
k+1) },\dots\bigr) $} and \smash{$\mathbf{x}g=\bigl( \dots,x_{1}^{(
k+1) },x_{>}^{( k+1) },\dots\bigr) $} and by Lemma~\ref{zeroD} $\rho\Delta( \mathbf{x}) =0$. Otherwise $k\in$ $E$
implies $k+1\in E$ or $k=\ell$ which by hypothesis implies $\ell\in E$ and
$1\notin E$, then \smash{$\mathbf{x=}\bigl( x_{>}^{(1) }\dots
,x_{1}^{(\ell) },\dots\bigr)$}, \smash{$\mathbf{x}g=\bigl(
x_{>}^{(1) }\dots,x_{1}^{(1) },\dots\bigr)$}
and $\rho\Delta( \mathbf{x}) =0$ as~before.
\end{proof}

It remains to compute $\rho g\phi_{E}$ for $E=\varnothing$ and $E=[
1,\ell] $. Note $g\phi_{\varnothing}=\phi_{\varnothing}$.

Suppose $\mathbf{x}\in X_{\varnothing,\varnothing}$, then \smash{$\mathbf{x}=\bigl(
x_{>}^{(1) },\dots,x_{>}^{(\ell) },x_{\ast
}^{( \ell+1) },\dots,x_{\ast}^{( b+1) }\bigr) $},
and since $\rho\Delta( \mathbf{x}g) =\rho\Delta(
\mathbf{x}) =\frac{1}{\pi_{p}}\psi$ (by Proposition~\ref{Lxh=Lx}) and
\smash{$\#X_{\varnothing,\varnothing}=\prod_{i=1}^{\ell}( n_{i}-1)
\prod_{j=\ell+1}^{b+1}n_{j}$}, it follows that
\[\rho\phi_{\varnothing
,\varnothing}=\frac{1}{\pi_{\ell}}\prod_{i=1}^{\ell}(
n_{i}-1) =\frac{1}{\pi_{\ell}}e_{\ell}( n_{\ast}-1).\]

Now suppose $E=[ 1,\ell] $ and $\mathbf{x}\in X_{\varnothing
,[ 1,\ell] }$ implies $\mathbf{x}=\bigl( x_{1}^{(
1) },\dots,x_{1}^{(\ell) },x_{\ast}^{(
\ell+1) },\dots,x_{\ast}^{( b+1) }\bigr) $, then
\[
\mathbf{x}g=\bigl( x_{1}^{(\ell) },x_{1}^{(1)
},\dots,x_{1}^{( \ell-1) },x_{\ast}^{( \ell+1)
},\dots,x_{\ast}^{( b+1) }\bigr).
\]
Applying $\ell-1$
transpositions $( \ell-1,\ell) ,( \ell-2,\ell-1)
,\dots,( 1,2) $ transforms $\mathbf{x}$ to $\mathbf{x}g$ and
thus \smash{$\Delta( \mathbf{x}g) =( -1) ^{\ell-1}%
\Delta( \mathbf{x}) $}. So
\[
\rho\Delta( \mathbf{x}g)
=( -1) ^{\ell-1}\rho\Delta( \mathbf{x}) =(
-1) ^{\ell-1}\frac{1}{\pi_{p}}\psi.
\] Since \smash{$\#X_{\varnothing,[
1,\ell] }=\prod_{i=\ell+1}^{b+1}n_{i}$}, it follows that \smash{$\rho
\phi_{B,[ 1,\ell] }=\frac{( -1) ^{\ell-1}}%
{\pi_{\ell}}$}.

\begin{Proposition}
Suppose $p=b+1$ and $2\leq\ell\leq b+1$, then
\[
\Phi^{\tau}( g)
=\frac{1}{\pi_{\ell}}\bigl\{ e_{\ell}( n_{\ast}-1) +(
-1) ^{\ell-1}\bigr\} .
\]
\end{Proposition}

\begin{proof}
$\rho g\psi=\rho g\phi_{\varnothing,\varnothing}+\rho g\phi_{\varnothing,[
1,\ell] }=\frac{1}{\pi_{\ell}}\bigl\{ e_{\ell}( n_{\ast
}-1) +( -1) ^{\ell-1}\bigr\} \psi$.
\end{proof}

This is the main formula (\ref{big1}) specialized to $m=0$.

\section[The cases p>b+1]{The cases $\boldsymbol{p>b+1}$}\label{pgtb1}

There is some simplification for $p=b+2$ compared to $p\geq b+3$. First, we set
up some tools.

\begin{Definition}
For an invariant basis element $\xi$ and a polynomial $\phi$, let
$\mathrm{coef}( \xi,\rho\phi) $ denote the coefficient of $\xi$
in the expansion of $\rho\phi$ in the basis.
\end{Definition}

The main object is to determine%
\begin{equation}
\Phi^{\tau}( g) =\sum_{S}\mathrm{coef}( \xi_{S},\rho
g\xi_{S}). \label{Phisum}%
\end{equation}

\begin{Proposition}
Suppose $S$ and $E$ are given by Definition~$\ref{defXE}$, $\ell\leq b+m$,
$\mathrm{coef}( \xi_{S},\rho g\phi_{S,E}) \neq0$, then
$E=\varnothing$ or $S\cap[ 1,\ell] =\varnothing$ and $E=[
1,\ell] $.
\end{Proposition}

\begin{proof}
Let $\upsilon E:=\{ j\in E\mid j+1\notin E\} $, the upper end-points
of $E$. If $j\in\upsilon E$ and ${j+1}\in{[ 1,b+m]} \setminus(
S\cup E) $, then $\mathcal{L}( \mathbf{x}g) =(
\dots,j+1,j+1,\dots) $ and $\rho g\Delta( \mathbf{x})
=0$ (Proposition~\ref{zeroD}) Thus ${\mathbf{x}\in X_{S,E}}$, $\rho
g\Delta( \mathbf{x}) \neq0$ and $k\in\upsilon E$ implies $k+1\in
S\cup\{ \ell\} $. Suppose $j\in\upsilon E$, $j+1\in S$ and~${j<\ell}$,
then \smash{$x_{1}^{( j+1) }$} appears in $g\phi_{E}$. That is,
if $x\in X_{S,E}$, then $j+1$ is not an entry of~$\mathcal{L}(
\mathbf{x}) $ but $j+1$ appears in $\mathcal{L}( \mathbf{x}%
g) $) and thus $\mathrm{coef}( \xi_{S},\rho g\phi_{S,E})
=0$. Another possibility is that there exists $i\notin S\cup E$, $1\leq i<\ell$
and $i+1\in E$, in which case $i+1$ does not appear in $\mathcal{L}(
\mathbf{x}g) $ and~${\mathrm{coef}( \xi_{S},\rho g\phi
_{S,E}) =0}$.
\end{proof}

The case $\ell=b+m+1$ involves more technicalities.

Informally, consider $S$ as the set of holes in $\mathcal{L}(
\mathbf{x}) $; no new holes can be adjoined or removed from
$\mathcal{L}( \mathbf{x}g) $ because this would imply
$\mathrm{coef}( \xi_{S},\rho\Delta( \mathbf{x}g) )
=0$. And of course $\mathcal{L}( \mathbf{x}g) $ can have no
repetitions. This is the idea that limits the possible boundary points of $E$
(that is, $j\in E$ and~${j+1\notin E}$ or $j-1\notin E$).

Recall $\phi_{S,E}:=\sum\{ \Delta( \mathbf{x})
\mid \mathbf{x}\in X_{S,E}\} $, and the task is to determine $\mathrm{coef}%
( \xi_{S},\rho g\phi_{S,E}) $.

\subsection[Case p=b+2]{Case $\boldsymbol{p=b+2}$}

Here $m=1$ so the sets $S$ are singletons $\{ i\} $ with $1\leq
i\leq b+1$. \emph{Note that} $m=1$ \emph{is an underlying hypothesis
throughout this subsection}. Write $\xi_{i}$ in place of $\xi_{\{
i\} }$. Then $\{ \xi_{i}\mid 1\leq i\leq b+1\} $ is a basis
for the $G_{\mathbf{n}}$-invariants. The possibilities for $E$ are $\varnothing$
for any $i$, $[ 1,\ell] $ for $i>\ell$, and $[
1,i-1] \cup\{ b+2\} $ for $\ell=b+2$. Suppose
$E=\varnothing$, then $\mathbf{x}\in X_{i,\varnothing}$ implies
\[
\mathbf{x}=\bigl(
x_{>}^{(1) },\dots,x_{>}^{(\ell) },x_{\ast
}^{( \ell+1) },\dots,x_{\ast}^{( b+2) }\bigr)
\]
with~\smash{$x_{>}^{(i) }$},~\smash{$x_{\ast}^{(i) }$} omitted if
$i\leq\ell$ or $i>\ell$, respectively. From Proposition~\ref{Lxh=Lx},
\[
\rho\Delta( \mathbf{x}g) =\rho\Delta( \mathbf{x})
=\frac{n_{i}}{\pi_{p}}\xi_{i}.
\]
 Also \smash{$\#X_{i,\varnothing}=\prod_{j=1,j\neq i}^{\ell}( n_{j}-1) \prod_{k=\ell
+1}^{b+2}n_{k}$}, or \smash{$\prod_{j=1}^{\ell}( n_{j}-1)
\prod_{k=\ell+1,k\neq i}^{b+2}n_{k}$}, if $i\leq\ell$ or $i>\ell$, respectively.

\begin{Proposition}
Suppose $2\leq\ell\leq b+1$ and $i>\ell$, then $\mathrm{coef}( \xi
_{i},\rho\phi_{i,\varnothing}) =\frac{1}{\pi_{\ell}}e_{\ell}(
n_{\ast}-1)$.
\end{Proposition}

\begin{proof}
Multiply $\#X_{i,\varnothing}$ by $\frac{n_{i}}{\pi_{p}}$ with result
$\frac{1}{\pi_{\ell}}\prod_{j=1}^{\ell}( n_{j}-1) $.
\end{proof}

\begin{Proposition}
\label{typei0}Suppose $2\leq\ell\leq b+2$ and $i\leq\ell$, then $\mathrm{coef}%
( \xi_{i},\rho g\phi_{i,\varnothing}) =\frac{1}{\pi_{\ell}}%
e_{\ell}( n_{\ast}-1) \bigl( 1+\frac{1}{n_{i}-1}\bigr).$
\end{Proposition}

\begin{proof}
Multiply $\#X_{i,\varnothing}$ by $\frac{n_{i}}{\pi_{p}}$ with result
\[
\prod_{j=1,j\neq i}^{\ell}\frac{n_{j}-1}{n_{j}}=\Biggl(
\prod_{j=1}^{\ell}\frac{n_{j}-1}{n_{j}}\Biggr) \biggl( \frac{n_{i}%
}{n_{i}-1}\biggr) =\frac{1}{\pi_{\ell}}\prod\limits_{j=1}^{\ell}(
n_{j}-1) \biggl( 1+\frac{1}{n_{i}-1}\biggr).
\tag*{\qed}
\]
\renewcommand{\qed}{}
\end{proof}

\begin{Proposition}
Suppose $2\leq\ell\leq b+2$ and $\ell<i$, then $\mathrm{coef}\bigl( \xi
_{i},\rho g\phi_{i,[ 1,\ell] }\bigr) =\frac{1}{\pi_{\ell}%
}( -1) ^{\ell-1}$.
\end{Proposition}

\begin{proof}
If $\mathbf{x}\in X_{i,[ 1,\ell] }$, then \smash{$\mathbf{x}=\bigl(
x_{1}^{(1) },\dots,x_{1}^{(\ell) },x_{\ast
}^{( \ell+1) },\dots,x_{\ast}^{( b+2) }\bigr) $}
omitting \smash{$x_{\ast}^{(i) }$} and $\mathcal{L}(
\mathbf{x}g) =( 2,3,\dots,\ell,1,\ell+1,\dots,i-1,i+1,\dots
,b+2) $. Applying a product of $\ell-1$ transpositions shows that
$\Delta( \mathbf{x}g) =( -1) ^{\ell-1}\Delta(
\mathbf{x}) $ and \smash{$\rho\Delta( \mathbf{x}g) =(
-1) ^{\ell-1}\frac{n_{i}}{\pi_{p}}\xi_{i}$}. Multiply \smash{$(
-1) ^{\ell-1}\frac{n_{i}}{\pi_{p}}$} by $\#X_{i,[ 1,\ell]
}=\smash{\prod_{s=\ell+1,s\neq i}^{b+2}n_{s}}$ to obtain \smash{$( -1)
^{\ell-1}\prod_{s=1}^{\ell}n_{s}^{-1}$}.
\end{proof}

\begin{Theorem}
Suppose $2\leq\ell\leq b+1$, then
\[
\Phi^{\tau}( g) =\frac{1}{\pi_{\ell}}\bigl\{ ( b+1)
e_{\ell}( n_{\ast}-1) +e_{\ell-1}( n_{\ast}-1)
+( -1) ^{\ell-1}( b-\ell+1) \bigr\}.
\]
\end{Theorem}

\begin{proof}
Break up the sum (\ref{Phisum}) into $i>\ell$ and $i\leq\ell$ sums:%
\begin{gather*}
\sum_{i=\ell+1}^{b+1}\mathrm{coef}( \xi_{i},\rho g\xi_{i})
=\sum_{i=\ell+1}^{b+1}\bigl( \mathrm{coef}( \xi_{i},\rho g\phi_{i,\varnothing}) +\mathrm{coef}\bigl( \xi_{i},\rho g\phi_{[1,\ell] }\bigr) \bigr) \\ \hphantom{\sum_{i=\ell+1}^{b+1}\mathrm{coef}( \xi_{i},\rho g\xi_{i})}{}
 =\frac{1}{\pi_{\ell}}( b+1-\ell) \bigl( e_{\ell}(n_{\ast}-1) +( -1) ^{\ell-1}\bigr) ,\\
\sum_{i=1}^{\ell}\mathrm{coef}( \xi_{i},\rho g\xi_{i})
=\sum_{i=1}^{\ell}\mathrm{coef}( \xi_{i},\rho g\phi_{i,\varnothing
}) =\frac{1}{\pi_{\ell}}\sum_{i=1}^{\ell}e_{\ell}( n_{\ast
}-1) \biggl( 1+\frac{1}{n_{i}-1}\biggr) \\ \hphantom{\sum_{i=1}^{\ell}\mathrm{coef}( \xi_{i},\rho g\xi_{i})}{}
 =\frac{1}{\pi_{\ell}}( \ell e_{\ell}( n_{\ast}-1)
+e_{\ell-1}( n_{\ast}-1) ).
\end{gather*}
Add the two parts together.
\end{proof}

\begin{Proposition}
Suppose $\ell=b+2$ and $1\leq i\leq b+1$, then for $E=[ 1,i-1]
\cup\{ b+2\} $%
\[
\mathrm{coef}( \xi_{i},\rho g\phi_{i,E}) =\frac{1}{\pi_{p}%
}( -1) ^{i-1}\prod_{s=i+1}^{b+1}( n_{s}-1).
\]
\end{Proposition}

\begin{proof}
If $\mathbf{x}\in X_{i,E}$, then \smash{$\mathbf{x}g=\bigl( x_{1}^{(2)
},\dots,x_{1}^{(i) },x_{>}^{( i+1) },\dots
,x_{1}^{(1) }\bigr) $}, $\mathcal{L}( \mathbf{x}g)
=( 2,\dots,i,i+1,\dots,\allowbreak {b+1},1) $ and $\Delta(
\mathbf{x}g) =( -1) ^{b}\Delta( \mathbf{y}) $
with $\mathcal{L}( \mathbf{y}) =( 1,2,\dots,b+1) $.
Apply Lemma~\ref{sumdelt} to obtain
\[
\sum_{j=1}^{b}( -1) ^{j}\Delta( x_{1},x_{2,},\dots
,\widehat{x_{j}},\dots,x_{b+2}) +( -1) ^{b+2}%
\Delta( x_{1},x_{2,},\dots,x_{b+1}) =0
\]
(the notation $\widehat{x_{j}}$ means $x_{j}$ is omitted). Use the term $j=i$
in the identity to obtain
\[\mathrm{coef}( \xi_{i},\rho g\Delta(
\mathbf{y}) ) =\frac{n_{i}}{\pi_{p}}( -1)
^{b+1-i}.\]
From $\#X_{i,E}=\prod_{s=i+1}^{b+1}( n_{s}-1) $, it
follows that
\[
\mathrm{coef}( \xi_{i},\rho g\phi_{i,E}) =\frac
{1}{\pi_{p}}( -1) ^{i-1}\prod_{s=i+1}^{b+1}( n_{s}%
-1).
\tag*{\qed}
\]
\renewcommand{\qed}{}
\end{proof}

\begin{Proposition}
\label{caselb2}Suppose $\ell=b+2$, then
\[
\sum_{i=1}^{b+1}\mathrm{coef}\bigl( \xi_{i},\rho g\phi_{i,[
1,i-1] \cup\{ p\} }\bigr) =\frac{1}{\pi_{p}}\Biggl\{
\prod_{s=1}^{b+1}( n_{s}-1) -( -1) ^{b}\Biggr\}.
\]

\end{Proposition}

\begin{proof}
The sum is%
\begin{align*}
\frac{1}{\pi_{p}}\sum_{i=1}^{b+1}( -1) ^{i-1}n_{i}\prod
_{s=i+1}^{b+1}( n_{s}-1) &{}=\frac{1}{\pi_{p}}\sum_{i=1}%
^{b+1}( -1) ^{i-1}( n_{i}-1+1) \prod_{s=i+1}%
^{b+1}( n_{s}-1) \\
&{}=\frac{1}{\pi_{p}}\sum_{i=1}^{b+1}\Biggl\{ ( -1) ^{i-1}%
\prod_{s=i}^{b+1}( n_{s}-1) -( -1) ^{i}\prod
_{s=i+1}^{b+1}( n_{s}-1) \Biggr\} \\
&{}=\frac{1}{\pi_{p}}\prod_{s=1}^{b+1}( n_{s}-1) -\frac{1}{\pi_{p}%
}( -1) ^{b+1}%
\end{align*}
by telescoping, leaving the first product with $i=1$ and the last with
$i=b+1.$
\end{proof}

\begin{Theorem}
\label{thmlb2}Suppose $\ell=b+2$, then
\[
\Phi^{\tau}( g) =\frac{1}{\pi_{\ell}}\bigl\{ (
\ell-1) e_{\ell}( n_{\ast}-1) +e_{\ell-1}( n_{\ast
}-1) +( -1) ^{b}\bigr\}.
\]
\end{Theorem}

\begin{proof}
Combine Propositions \ref{typei0} and \ref{caselb2} (note $\pi_{p}=\pi_{\ell}$),
\begin{align*}
\sum_{i=1}^{b+1}\mathrm{coef}( \xi_{i},\rho g\xi_{i}) &
=\frac{1}{\pi_{\ell}}\sum_{i=1}^{\ell-1}\prod\limits_{j=1}^{\ell}(
n_{j}-1) \biggl( 1+\frac{1}{n_{i}-1}\biggr) +\frac{1}{\pi_{p}}%
\prod_{s=1}^{\ell-1}( n_{s}-1) -\frac{1}{\pi_{p}}(
-1) ^{b+1}\\
& =\frac{1}{\pi_{\ell}}\prod\limits_{j=1}^{\ell}( n_{j}-1)
\Biggl\{ \sum_{i=1}^{\ell-1}\biggl( 1+\frac{1}{n_{i}-1}\biggr) +\frac
{1}{n_{\ell}-1}\Biggr\} -\frac{1}{\pi_{p}}( -1) ^{b+1}\\
& =\frac{1}{\pi_{\ell}}\bigl\{ ( \ell-1) e_{\ell}(
n_{\ast}-1) +e_{\ell-1}( n_{\ast}-1) +( -1)
^{b}\bigr\}.
\tag*{\qed}
\end{align*}
\renewcommand{\qed}{}
\end{proof}

Formula (\ref{big1}) with $m=1$, $\ell=b+2$ has the term \smash{$( -1)
^{\ell+1}\frac{( b-\ell+1) _{m}}{m!}=( -1)
^{\ell+2}=( -1) ^{b}$}. This completes the case $p=b+2.$

\subsection[Case p>b+2]{Case $\boldsymbol{p>b+2}$}

Write $p=b+m+1$. Label the invariants by $S\subset[ 1,b+m]$, $\#S=m$,
\begin{gather*}
\xi_{S}:=\sum\bigl\{
\Delta\bigl( x_{i_{1}}^{( j_{1}) },x_{i_{2}}^{(
j_{2}) },\dots,x_{i_{b}}^{( j_{b}) },x_{i_{p}}^{(
p) }\bigr) \mid \{ j_{1},\dots,j_{b}\} =[
1,b+m] \setminus S,\\
\hphantom{\xi_{S}:=\sum\bigl\{
\Delta\bigl( x_{i_{1}}^{( j_{1}) },x_{i_{2}}^{(
j_{2}) },\dots,x_{i_{b}}^{( j_{b}) },x_{i_{p}}^{(
p) }\bigr) \mid} \
1\leq i_{s}\leq n_{j_{s}},\, 1\leq s\leq b,\, 1\leq i_{p}\leq n_{p}
\bigr\},
\end{gather*}
and in $\Delta( \mathbf{x}) $ take $j_{1}<j_{2}<\dots<j_{b}$.
The following lemma generalizes the generating function for elementary
symmetric polynomials.

\begin{Lemma}
\label{prodsm}Suppose $y_{1},y_{2},\dots,y_{r}$ are variables and $q\leq
r\leq s$, then
\[
\prod\limits_{i=1}^{r}y_{i}\sum_{U\subset[ 1,s] ,\, \#U=q}
\prod\limits_{j\in U\cap[ 1,r] }\biggl( 1+\frac{1}{y_{j}}\biggr)
=\sum_{k=0}^{\min( q,r) }\binom{s-k}{q-k}e_{r-k}(
y_{1},\dots,y_{r}).
\]
\end{Lemma}

\begin{proof}
The product
\[
\prod\limits_{j\in U\cap[ 1,r] }\biggl( 1+\frac
{1}{y_{j}}\biggr) =\sum\limits_{k=0}^{q}\sum\biggl\{ \prod\limits_{j\in
V}\biggl( \frac{1}{y_{j}}\biggr) \mid V\subset U\cap[ 1,r]
,\, \#V=k \biggr\}.
\]
Any particular $V$ with $\#V=k$ appears in $\binom
{s-k}{q-k}$ different sets $U$. Then $e_{k}\bigl( y_{1}^{-1},\dots
,y_{r}^{-1}\bigr) $ is a~sum of \smash{$\prod_{j\in V}\bigl( \frac{1}%
{y_{j}}\bigr) $} over $k$-subsets of $[ 1,r] $, and thus the sum
is \smash{$\sum_{k=0}^{\min( q,r) }\binom{s-k}{q-k}e_{k}\bigl(
y_{1}^{-1},\dots,y_{r}^{-1}\bigr) $}. Also
\[
\left( \prod_{i=1}^{r}y_{i}\right) e_{k}\bigl( y_{1}^{-1},\dots,y_{r}^{-1}\bigr)
=e_{r-k}( y_{1},\dots,y_{r}).\tag*{\qed}
\]\renewcommand{\qed}{}
\end{proof}

The apparent singularity at $y_{j}=0$ is removable.

\begin{Proposition}
\label{m+1case}If $S\subset[ 1,b+m]$, $\#S=m$ and $\ell\leq b+m+1$,
then
\[
\mathrm{coef}( \xi_{S},\rho g\phi_{S,\varnothing}) =\prod\biggl\{
\frac{n_{i}-1}{n_{i}}\mid 1\leq i\leq\ell,\,i\notin S\biggr\}.
\]
\end{Proposition}

\begin{proof}
When $E=\varnothing$, then $\mathbf{x}\in X_{S,E}$ satisfies $\rho\Delta(
\mathbf{x}g) =\rho\Delta( \mathbf{x}) =\bigl(
{\prod _{j=1,j\notin S}^{b+m+1}}
n_{j}^{-1}\bigr) \xi_{S}$. Furthermore, \smash{$\#X_{S,\varnothing}=%
{\prod _{i=1,i\notin S}^{\ell}}
( n_{i}-1) \times
{\prod _{j=\ell+1,j\notin S}^{b+m+1}}
n_{j}$} and the product of the two factors is
\[
{\prod_{i=1,i\notin S}^{\ell}}
\biggl( \frac{n_{i}-1}{n_{i}}\biggr).
\tag*{\qed}
\]
\renewcommand{\qed}{}
\end{proof}

\begin{Proposition}
For $\ell\leq b+m$,
\[
\sum_{S\subset[ 1,b+m] ,\,\#S=m}\mathrm{coef}( \xi_{S},\rho
g\phi_{S,\varnothing}) =\frac{1}{\pi_{\ell}}\sum_{k=0}^{\min(
m,\ell) }\frac{( b+1) _{m-k}}{( m-k)
!}e_{\ell-k}( n_{\ast}-1).
\]

\end{Proposition}

\begin{proof}
The sum equals%
\begin{align*}
\sum_{S\subset[ 1,b+m] ,\,\#S=m}%
{\prod\limits_{s=1,i\notin S}^{\ell}}
\biggl( \frac{n_{s}-1}{n_{s}}\biggr) & =\frac{1}{\pi_{\ell}}%
\prod\limits_{i=1}^{\ell}( n_{i}-1) \sum_{S\subset[
1,b+m] ,\,\#S=m}%
{\prod\limits_{j\in[ 1,\ell] \cap S}^{\ell}}
\biggl( \frac{n_{j}}{n_{j}-1}\biggr) \\
& =\frac{1}{\pi_{\ell}}\sum_{k=0}^{\min( m,\ell) }\binom
{b+m-k}{m-k}e_{\ell-k}( n_{\ast}-1)
\end{align*}
by Lemma~\ref{prodsm} with $r=\ell$, $s=b+m$, $q=m$ and $y_{i}=n_{i}-1$. Also
\[
\binom{b+m-k}{m-k}=\frac{( b+m-k) !}{b!( m-k)
!}=\frac{( b+1) _{m-k}}{( m-k) !}.
\tag*{\qed}
\]
\renewcommand{\qed}{}
\end{proof}

\begin{Proposition}%\label{ellS}
If $\ell\leq b,$ $S\subset[ \ell+1,b+m]$, $\#S=m$, then
\smash{$\mathrm{coef}\bigl( \xi_{S},\rho g\phi_{S,[ 1,\ell] }\bigr)
=\frac{( -1) ^{\ell+1}}{\pi_{\ell}}$}.
\end{Proposition}

\begin{proof}
If $\mathbf{x}\in X_{S,[ 1,\ell] }$, then $\mathcal{L}(
\mathbf{x}g) =( 2,3,\dots,\ell,1,\ell+1,\dots,b+m+1) $
with $\{ j\mid j\in S\} $ omitted. Thus
\[
\Delta( \mathbf{x}%
g) =( -1) ^{\ell-1}\Delta( \mathbf{x}) \qquad \text{and}\qquad
\rho\Delta( \mathbf{x}g) =( -1) ^{\ell-1}%
{\prod_{i=1,i\notin S}^{b+m+1}}
n_{i}^{-1}\xi_{S}.
\]
 The number of summands in $\phi_{S,[ 1,\ell]
}$ is
\[
\#X_{S,[ 1,\ell] }=
{\prod_{s=\ell+1,s\notin S}^{b+m+1}}
n_{s}
\] and the required coefficient is the product with \smash{$( -1)
^{\ell-1}
{\prod_{i=1,i\notin S}^{b+m+2}}
n_{i}^{-1}$}, namely
\[
( -1) ^{\ell-1}%
{\prod_{i=1}^{\ell}}
n_{i}^{-1}.\tag*{\qed}
\]\renewcommand{\qed}{}
\end{proof}

\begin{Theorem}
Suppose $\ell\leq b+m$, then
\[
\Phi^{\tau}( g) =\frac{1}{\pi_{\ell}}\Biggl\{ \sum_{k=0}%
^{\min( m,\ell) }\frac{( b+1) _{m-k}}{(
m-k) !}e_{\ell-k}( n_{\ast}-1) +( -1)
^{\ell+1}\frac{( b-\ell+1) _{m}}{m!}\Biggr\}.
\]
\end{Theorem}

\begin{proof}
When $b<\ell\leq b+m$, then the sum (\ref{Phisum}) equals $\sum_{S}%
\mathrm{coef}( \xi_{S},\rho g\phi_{S,\varnothing}) $, else if
${2\leq\ell\leq b}$, then it equals
\[
\sum_{S\subset[ 1,b+m]
}\mathrm{coef}( \xi_{S},\rho g\phi_{S,\varnothing}) +\sum_{S\subset[ \ell+1,b+m] }\mathrm{coef}( \xi_{S},\rho
g\phi_{S,[ 1,\ell] }).
\]
 There are \smash{$\binom{b+m-\ell}{m}$}
subsets $S\subset[ \ell+1,b+m] $. In both cases the sums evaluate
to the claimed~value, since \smash{$\frac{( b-\ell+1) _{m}}{m!}%
=\binom{b+m-\ell}{m}$} if $\ell\leq b$ and $=0$ if $b+1\leq\ell\leq b+m$.
\end{proof}

For the case $\ell=b+m+1$, the sets $E$ which allow $\mathrm{coef}(
\xi_{S},\rho g\phi_{S,E}) \neq0$ are $E=\varnothing,[ 1,\min
S-1] \cup\{ \ell\} $. Lemma~\ref{sumdelt} is used just as
in the situation $m=1$.

\begin{Proposition}
\label{Enull}Suppose $\ell=b+m+1$, then
\begin{gather*}
\sum_{S\subset[ 1,b+m] ,\#S=m}\mathrm{coef}( \xi_{S},\rho
g\phi_{S,\varnothing}) =\frac{1}{\pi_{\ell}}( n_{\ell}-1)
\sum_{k=0}^{m}\frac{( b+1) _{m-k}}{( m-k) !}%
e_{\ell-1-k}( n_{1}-1,\dots,n_{\ell-1}-1).
\end{gather*}
\end{Proposition}

\begin{proof}By Proposition~\ref{m+1case},
\[
\mathrm{coef}( \xi_{S},\rho g\phi_{S,\varnothing}) =\prod
_{i=1,i\notin S}^{b+m+1}\frac{n_{i}-1}{n_{i}}=\Biggl( \prod_{i=1}^{b+m+1}%
\frac{n_{i}-1}{n_{i}}\Biggr) \prod\limits_{j\in S}\biggl( \frac{n_{j}}%
{n_{j}-1}\biggr)
\]
and%
\begin{align*}
\sum_{S\subset[ 1,b+m] ,\#S=m}\mathrm{coef}( \xi_{S},\rho
g\phi_{S,\varnothing}) &{}=\frac{1}{\pi_{\ell}}\prod\limits_{i=1}^{\ell
}( n_{i}-1) \sum_{S\subset[ 1,b+m] }\prod
\limits_{j\in S}\biggl( 1+\frac{1}{n_{j}-1}\biggr) \\
&{}=( n_{\ell}-1) \frac{1}{\pi_{\ell}}\prod\limits_{i=1}%
^{b+m}( n_{i}-1) \sum_{S\subset[ 1,b+m] }%
\prod\limits_{j\in S}\biggl( 1+\frac{1}{n_{j}-1}\biggr) \\
&{}=\frac{1}{\pi_{\ell}}( n_{\ell}-1) \sum_{k=0}^{m}\frac{(
b+1) _{m-k}}{( m-k) !}e_{\ell-1-k}( n_{1}%
-1,\dots,n_{\ell-1}-1) ,
\end{align*}
by Lemma~\ref{prodsm} with $r=s=b+m$, $q=m$, ($r=\ell-1$) and $y_{i}=n_{i}-1$.
\end{proof}

\begin{Proposition}
\label{minS0}Suppose $S\subset[ 1,b+m]$, $\#S=m$ and $\ell=b+m+1$,
and $E=[ 1,t-1] \cup\{ \ell\} $ with $t:=\min S$,
then
\[
\mathrm{coef}( \xi_{S},\rho g\phi_{S,E}) =( -1)
^{t+1}\frac{1}{\pi_{\ell}}\prod\limits_{i=t}^{b+m}( n_{i}-1)
\prod\limits_{j\in S}\biggl( 1+\frac{1}{n_{j}-1}\biggr).
\]
\end{Proposition}

\begin{proof}
If $\min S=1$, then $E=\{ \ell\} $. Set $t:=\min S$. A typical
point in $X_{S,E}$ is
\[\mathbf{x=}\bigl( x_{1}^{(1) }%
,x_{1}^{(2) },\dots,x_{1}^{( t-1) },x_{>}^{(
t+1) },\dots,x_{>}^{( \ell-1) },x_{1}^{(
\ell) }\bigr)\]
omitting \smash{$\bigl\{ x_{\ast}^{( j) }\mid j\in
S\bigr\} $}. Then
\begin{gather*}
\mathbf{x}g=\bigl( x_{1}^{(2) }%
,x_{1}^{( 3) },\dots,x_{1}^{( t) },x_{>}^{(
t+1) },\dots,x_{>}^{( \ell-1) },x_{1}^{(1)
}\bigr) \qquad \text{and} \\
\Delta( \mathbf{x}g) =(
-1) ^{b}\Delta\bigl( x_{1}^{(1) },x_{1}^{(
2) },\dots,x_{1}^{( t) },x_{>}^{( t+1)
},\dots,x_{>}^{( \ell-1) }\bigr)
\end{gather*}
 omitting $S\setminus
\{ t\} $ terms (applying $b$ adjacent transpositions). To apply
Lemma~\ref{sumdelt}, we relabel \smash{$\bigl( x_{1}^{(1) }%
,\dots,x_{1}^{( t) },x_{>}^{( t+1) },\dots
,x_{>}^{( \ell-1) },x_{1}^{(\ell) }\bigr) $} (omit
$S\setminus\{ t\} $) as $( y_{1},\dots,y_{b+2}) $
with \smash{$y_{i}=x_{1}^{(i) }$} for $1\leq i\leq t$ and $i=\ell$.
Thus
\[
\Delta( y_{1},y_{2,},\dots,y_{b+1}) =\sum_{j=1}^{b}(
-1) ^{j+b+1}\Delta( y_{1},y_{2,},\dots,\widehat{y_{j}}%
,\dots,y_{b+2}).
\]
Apply $\rho$, then the term with $j=t$ becomes $( -1)
^{t+b+1}\bigl( \prod_{i=1,i\notin S}^{\ell}n_{i}^{-1}\bigr) \xi_{S}%
$. Thus%
\[
\mathrm{coef}( \xi_{S},\rho\Delta( \mathbf{x}g) )
=( -1) ^{b}( -1) ^{t+b+1}\Biggl( \prod
\limits_{i=1,i\notin S}^{\ell}n_{i}^{-1}\Biggr).
\]
Multiply by $\#X_{S,E}=\prod_{i=t+1,i\notin S}^{b+m}(
n_{i}-1) $ to obtain%
\[
\mathrm{coef}( \xi_{S},\rho g\phi_{S,E}) =( -1)
^{t+1}\frac{1}{\pi_{\ell}}\prod\limits_{i=t}^{b+m}( n_{i}-1)
\prod\limits_{j\in S}\biggl( 1+\frac{1}{n_{j}-1}\biggr).
\tag*{\qed}
\]
\renewcommand{\qed}{}
\end{proof}

The next step is to sum over $S$ with the same $\min S.$

\begin{Proposition}
Suppose $\ell=b+m+1$, $1\leq t\leq b+1$, and $E=[ 1,t-1]
\cup\{ \ell\} $, then
\[
\sum_{\min S=t}\mathrm{coef}( \xi_{S},\rho g\phi_{S,E}) =(
-1) ^{t}\frac{n_{t}}{\pi_{\ell}}\sum_{k=0}^{\ell-1-t}\frac{(
m-k) _{b+1-t}}{( b+1-t) !}e_{\ell-1-t-k}(
n_{t+1}-1,\dots,n_{\ell-1}-1).
\]

\end{Proposition}

\begin{proof}
By Proposition~\ref{minS0},
\begin{align*}
\sum_{\min S=t}\mathrm{coef}( \xi_{S},\rho g\phi_{S,E}) &{}
=( -1) ^{t+1}\frac{1}{\pi_{\ell}}\prod\limits_{i=t}^{b+m}(
n_{i}-1) \biggl( 1+\frac{1}{n_{t}-1}\biggr) \\
& \quad{}\times\sum_{U\subset[ t+1,b+m] ,\#U=m-1}\prod\limits_{j\in
U}\biggl( 1+\frac{1}{n_{j}-1}\biggr) \\
&{}=( -1) ^{t}\frac{n_{t}}{\pi_{\ell}}\sum_{k=0}^{m-1}%
\binom{b+m-t-k}{m-1-k}e_{\ell-1-t-k}( n_{t+1}-1,\dots,n_{\ell
-1}-1)
\end{align*}
by Lemma~\ref{prodsm} with $r=s=\ell-1-t$, $q=m-1$ (note $b+m=\ell-1$). In the
inner sum $S=U\cup\{ t\} $. The binomial coefficient is equal to
the coefficient in the claim.
\end{proof}

To shorten some ensuing expressions, introduce
\[
e( k;u) :=e_{k}( n_{u}-1,n_{u+1}-1,\dots,n_{\ell
-1}-1).
\]

\begin{Proposition}
\label{tele3}Suppose $\ell=b+m+1$, then
\begin{gather*}
\sum_{t=1}^{b+1}( -1) ^{t}n_{t}\sum_{k=0}^{m-1}\frac{(
m-k) _{b+1-t}}{( b+1-t) !}e( \ell-1-t-k;t+1)
=\sum_{k=1}^{m}\frac{( b+1) _{m-k}}{( m-k)
!}e( \ell-k;1) +( -1) ^{b}.
\end{gather*}

\end{Proposition}

\begin{proof}
Write $n_{t}=( n_{t}-1) +1$ and use a simple identity for
elementary symmetric functions
\[
n_{t}e( \ell-1-t-k;t+1) =e( \ell-t-k;t) -e(
\ell-t-k;t+1) +e( \ell-1-t-k;t+1),
\]
then the sum becomes
\begin{gather}
 \sum_{k=0}^{m-1}\frac{( m-k) _{b}}{b!}e( \ell
-1-k;1) \nonumber%\label{EE1}
\\
\qquad{} +\sum_{t=2}^{b+1}( -1) ^{t+1}\sum_{k=0}^{m-1}\frac{(
m-k) _{b+1-t}}{( b+1-t) !}e( \ell-t-k;t)
\label{EE2}\\
\qquad{} -\sum_{t=1}^{b+1}( -1) ^{t+1}\sum_{k=0}^{m-1}\frac{(
m-k) _{b+1-t}}{( b+1-t) !}e( \ell-t-k;t+1)
\label{EE3}\\
\qquad{} +\sum_{t=1}^{b+1}( -1) ^{t+1}\sum_{k=0}^{m-1}\frac{(
m-k) _{b+1-t}}{( b+1-t) !}e( \ell-1-t-k;t+1)
,\nonumber %\label{EE4}
\end{gather}
We will show that there is a three-term telescoping effect, after changing the
summation variables in sums: (\ref{EE2}) $t\rightarrow t+1$, (\ref{EE3})
$k\rightarrow k+1$. This results in (displayed in same order)%
\begin{align*}
& \sum_{k=0}^{m-1}\frac{( m-k) _{b}}{b!}e( \ell
-1-k;1) \\
&\qquad{} +\sum_{t=1}^{b+1}( -1) ^{t}\sum_{k=0}^{m-1}\frac{(
m-k) _{b-t}}{( b-t) !}e( \ell-1-t-k;t+1) \\
&\qquad{} -\sum_{t=1}^{b+1}( -1) ^{t+1}\sum_{k=-1}^{m-2}\frac{(
m-k-1) _{b+1-t}}{( b+1-t) !}e( \ell
-1-t-k;t+1) \\
&\qquad{} +\sum_{t=1}^{b+1}( -1) ^{t+1}\sum_{k=0}^{m-1}\frac{(
m-k) _{b+1-t}}{( b+1-t) !}e( \ell-1-t-k;t+1).
\end{align*}
Thus the coefficient of $e( \ell-1-t-k;t+1) $ is
\begin{align*}
& \sum_{t=1}^{b}( -1) ^{t}\sum_{k=0}^{m-1}\frac{(
m-k) _{b-t}}{( b-t) !}-\sum_{t=1}^{b+1}( -1)
^{t+1}\sum_{k=-1}^{m-2}\frac{( m-k-1) _{b+1-t}}{(
b+1-t) !}\\
&\qquad{} +\sum_{t=1}^{b+1}( -1) ^{t+1}\sum_{k=0}^{m-1}\frac{(
m-k) _{b+1-t}}{( b+1-t) !},
\end{align*}
the limits in the middle sum can be replaced by $0\leq k\leq m-2$ since
$e( \ell-t;t+1) =0$ (at $k=-1$). If a pair $( t,k) $
occurs in each sum, then the sum of these terms vanishes, by a~straightforward
calculation. Exceptions are at $t=b+1$ (where $\ell-1-b-1=m-1$ and~${e(
\ell-1-t-k;t+1) =e( m-1-k;b+2) }$) and at $1\leq t\leq
b$, $k=m-1$%
\begin{align*}
&( -1) ^{b+2}\Biggl\{ -\sum_{k=0}^{m-2}1+\sum_{k=0}^{m-1}%
1\Biggr\} e( m-1-k;b+2) =( -1) ^{b}e(
0;b+2), \\
&\Biggl\{ \sum_{t=1}^{b}( -1) ^{t}+\sum_{t=1}^{b}(
-1) ^{t+1}\Biggr\} e( \ell-t-m;t+1) =0,
\end{align*}
respectively. Thus%
\begin{gather*}
\sum_{t=1}^{b+1}( -1) ^{t}n_{t}\sum_{k=0}^{m-1}\frac{(
m-k) _{b+1-t}}{( b+1-t) !}e( \ell-1-t-k;t+1)
\\
\qquad{} =\sum_{k=0}^{m-1}\frac{( m-k) _{b}}{b!}e( \ell
-1-k;1) +( -1) ^{b}\\
\qquad{} =\sum_{k=1}^{m}\frac{( b+1) _{m-k}}{( m-k)
!}e( \ell-k;1) +( -1) ^{b},
\end{gather*}
(changing $k\rightarrow k-1$) since $\frac{( m-k) _{b}}%
{b!}=\binom{m-k-1+b}{m-k-1}=\frac{( b+1) _{m-k-1}}{(
m-k-1) !}$.
\end{proof}

\begin{Theorem}
\label{thmlb3}Suppose $\ell=b+m+1( =p) $, then
\[
\Phi^{\tau}( g) =\pi_{\ell}^{-1}\Biggl\{ \sum_{k=0}^{m}%
\frac{( b+1) _{m-k}}{( m-k) !}e_{\ell-k}(
n_{\ast}-1) +( -1) ^{b}\Biggr\}.
\]
\end{Theorem}

\begin{proof}
Combining the values from Proposition~\ref{Enull} for $E=\varnothing$ and from
Proposition~\ref{tele3} for $E=[ 1,\min S-1] \cup\{ p\} $,
\begin{align*}
\sum_{S}\mathrm{coef}( \xi_{S},\rho g\xi_{S}) & =\frac{1}%
{\pi_{\ell}}( n_{\ell}-1) \sum_{k=0}^{m}\frac{(
b+1) _{m-k}}{( m-k) !}e_{\ell-1-k}( n_{1}%
-1,\dots,n_{\ell-1}-1) \\
& \quad{}+\frac{1}{\pi_{\ell}}\Biggl\{ \sum_{k=1}^{m}\frac{( b+1)
_{m-k}}{( m-k) !}e_{\ell-k}( n_{1}-1,\dots,n_{\ell
-1}-1) +( -1) ^{b}\Biggr\} \\
& =\frac{1}{\pi_{\ell}}\Biggl\{ \sum_{k=0}^{m}\frac{( b+1)
_{m-k}}{( m-k) !}e_{\ell-k}( n_{\ast}-1) +(
-1) ^{b}\Biggr\}.
\end{align*}
In the second line the lower limit $k=1$ can be replaced by $k=0$ because
\[
e_{\ell}( n_{1}-1,\dots,n_{\ell-1}-1) =0.
\tag*{\qed}
\]
\renewcommand{\qed}{}
\end{proof}

Observe that formula (\ref{big1}) contains \smash{$( -1) ^{\ell+1}%
\frac{( b-\ell+1) _{m}}{m!}$} which becomes
\[
( -1)
^{\ell+1}\frac{( -m) _{m}}{m!}=( -1) ^{m+\ell+1},
\]
and $\ell=b+m+1$. Thus we have proven the general formula for any $\ell$ with
$2\leq\ell\leq p=b+m+1$.

\section{An equivalent formula}\label{formb2}

Formula (\ref{big1}) can be expressed in terms of $e_{k}\bigl( \frac{1}
{n_{1}},\dots,\frac{1}{n_{\ell}}\bigr) $, as displayed in formula~(\ref{big2}).

\begin{Proposition}
\label{eqbig23}For $m\geq0$ and $2\leq\ell\leq b+m+1=p$,
\begin{align*}
& \frac{1}{\pi_{\ell}}\Biggl\{ \sum_{k=0}^{\min( m,\ell) }%
\frac{( b+1) _{m-k}}{( m-k) !}e_{\ell-k}(
n_{\ast}-1) +( -1) ^{\ell+1}\frac{( b-\ell+1)
_{m}}{m!}\Biggr\} \\
& \qquad{} =\binom{b+m}{b}+\sum_{i=1}^{\min( b,\ell-1) }(
-1) ^{i}\binom{b+m-i}{b-i}e_{i}\biggl( \frac{1}{n_{1}},\dots,\frac
{1}{n_{\ell}}\biggr).
\end{align*}
\end{Proposition}

\begin{proof}
From the generating function for elementary symmetric functions (we denote
$e_{i}( n_{1},n_{2},\allowbreak \dots,n_{\ell}) $ by $e_{i}( n_{\ast
}) $),
\begin{align*}
\sum_{j=0}^{\ell}t^{j}e_{j}( n_{\ast}-1) & =\prod_{i=1}^{\ell
}( 1+t( n_{i}-1) ) =( 1-t) ^{\ell}%
\prod_{i=1}^{\ell}\biggl( 1+\frac{t}{1-t}n_{i}\biggr) \\
& =\sum_{i=1}^{\ell}( 1-t) ^{\ell-i}t^{i}e_{i}( n_{\ast
}) =\sum_{i=1}^{\ell}\sum_{k=0}^{\ell-i}( -1) ^{k}%
\binom{\ell-i}{k}t^{i+k}e_{i}( n_{\ast}) \\
& =\sum_{j=0}^{\ell}t^{j}\sum_{i=0}^{j}( -1) ^{j-i}\binom
{\ell-i}{j-i}e_{i}( n_{\ast}),
\end{align*}
and thus $e_{j}( n_{\ast}-1) =\sum\limits_{i=0}^{j}(
-1) ^{j-i}\binom{\ell-i}{j-i}e_{i}( n_{\ast}) $. The
first formula equals%
\[
\pi_{\ell}^{-1}\Biggl\{ \sum_{k=0}^{\min( m,\ell) }\sum
_{i=0}^{\ell-k}\frac{( b+1) _{m-k}}{( m-k) !}(
-1) ^{\ell-k-i}\binom{\ell-i}{\ell-k-i}e_{i}( n_{\ast})
+( -1) ^{\ell+1}\frac{( b-\ell+1) _{m}}{m!}\Biggr\}
;
\]
the coefficient of $e_{i}( n_{\ast}) $ is%
\[
\sum_{k=0}^{\min( m,\ell-i) }\frac{( b+1) _{m-k}%
}{( m-k) !}\frac{( i-\ell) _{k}}{k!}(
-1) ^{\ell-i}=( -1) ^{\ell-i}\frac{( b+1+i-\ell
) _{m}}{m!},
\]
(by the Chu--Vandermonde sum) which leads to%
\begin{align*}
& \pi_{\ell}^{-1}\Biggl\{ \sum_{i=0}^{\ell}( -1) ^{\ell-i}%
\frac{( b+1+i-\ell) _{m}}{m!}e_{i}( n_{\ast})
+( -1) ^{\ell+1}\frac{( b-\ell+1) _{m}}{m!}\Biggr\}
\\
&\qquad{} =\pi_{\ell}^{-1}\sum_{i=1}^{\ell}( -1) ^{\ell-i}\frac{(
b+1+i-\ell) _{m}}{m!}e_{i}( n_{\ast}) =\sum_{j=0}^{\ell
-1}( -1) ^{j}\frac{( b+1-j) _{m}}{m!}\frac
{e_{\ell-j}( n_{\ast}) }{\pi_{\ell}};
\end{align*}
(with $j=\ell-i$) this is the second formula since
\[
\frac{(
b+1-j) _{m}}{m!}=\frac{( b+m-j) !}{( b-j)
!m!}\qquad \text{and}\qquad \frac{e_{\ell-j}( n_{\ast}) }{\pi_{\ell}}%
=e_{j}\biggl( \frac{1}{n_{1}},\dots,\frac{1}{n_{\ell}}\biggr).
\tag*{\qed}
\]
\renewcommand{\qed}{}
\end{proof}

The second formula is more concise than the first one when $b$ is relatively
small. For example, when $b=1$ (the isotype $[ N-1,1] $ and
$p=m+2$), the value is $p-1-e_{1}\bigl( \frac{1}{n_{1}},\dots,\frac
{1}{n_{\ell}}\bigr) $; this was already found in \cite[Theorem~5.6]%
{DunklGorin2024}.

Another interesting specialization of the first formula is for $n_{i}=1$ for
all $i$ so that the spherical function reduces to the character \big($\tau=\bigl[
N-b,1^{b}\bigr] $\big)
\[
\chi^{\tau}( g) =
\begin{cases}
\displaystyle \frac{( b+1) _{m-\ell}}{( m-\ell) !}+(
-1) ^{\ell+1}\frac{( b-\ell+1) _{m}}{m!}, &\ell\leq m,\\
\displaystyle ( -1) ^{\ell+1}\frac{( b-\ell+1) _{m}}{m!}
,& m<\ell\leq N.
\end{cases}
\]
Observe that $\chi^{\tau}( g) =0$ when $b\leq\ell-1$ and $m\geq
1$. If $N=b+1$, $m=0$, then $\tau=\mathrm{sign}$ whose value at an $\ell$-cycle is
$( -1) ^{\ell+1}$.

\subsection*{Acknowledgements}
The author is grateful to the referees whose careful reading and detailed suggestions helped to improve this paper.

\pdfbookmark[1]{References}{ref}
\LastPageEnding

\end{document}